\documentclass[doublecol]{epl2}
\usepackage{amsfonts}
\usepackage{graphicx}
\usepackage{dcolumn}
\usepackage{amssymb}
\usepackage{array}
\usepackage{longtable}
\usepackage{latexsym, color}
\usepackage{amsthm}
\usepackage{graphicx,booktabs,multirow}
\usepackage[latin1]{inputenc}
\usepackage{amssymb}
\usepackage{paralist}
\usepackage{amsmath}
\usepackage{lipsum}
\usepackage{appendix}
\newtheorem{defi}{Definition}
\newtheorem{teo}[defi]{Theorem}

\newtheorem{cor}[defi]{Corollary}

\newtheorem{os}[defi]{Remark}

\title{Tutte polynomial of a small-world Farey graph}
\shorttitle{Tutte polynomial of a small-world Farey graph} 

\author{Yunhua Liao, Yaoping Hou \and Xiaoling Shen}
\shortauthor{Yunhua Liao \etal}

\institute{Department of Mathematics,
Hunan Normal University, Changsha, Hunan 410081, China
}

\abstract{In this paper, we find recursive formulas for the Tutte polynomial of a family of small-world networks: Farey graphs, which are modular and have an exponential degree hierarchy. Then, making use of these formulas, we determine the number of  spanning trees, as well as the number of connected spanning subgraphs. Furthermore, we also derive exact expressions for the chromatic polynomial and the reliability polynomial of these graphs.}

\begin{document}

\maketitle
\section{Introduction}
The Tutte polynomial of a graph, also known as the partition function of the Potts model \cite{Potts52,Wu82}, is a renown tool for analyzing properties of graphs and networks. The two-variable polynomial, due to W.T. Tutte \cite{Tutte54,Tutte67}, plays an important role in several areas of sciences, for instance, combinatorics, statistical mechanics and biology. In a strong sense it contains every graphical invariant that can be computed by deleting and contraction operations which are natural reductions for many networks model. The Tutte polynomial $T(G;x,y)$ can be evaluated at particular points $(x,y)$ to give numerical graph invariants, including the number of spanning trees, the number of connected spanning subgraphs, the dimension of the bicycle space and many more. The Tutte polynomial also specializes to a variety of single-variable graphical polynomials, including the chromatic polynomial, the reliability polynomial and the flow polynomial.
Furthermore, the Tutte polynomial has  been widely studied in the field of statistical physics where it appears as the partition function of the Potts model $Z_G(q,v)$.  In fact, if $G$ is a graph on $n$ vertices with $k$ connected components, then
\begin{equation*}
T(G;x,y)=(x-1)^{-k}(y-1)^{-n}Z_G((x-1)(y-1),(y-1))
\end{equation*}
and so the partition function of the  Potts model is simply the Tutte polynomial expressed in different variables.
 For overviews, see Refs.~\cite{Welsh00,Ellis11a,Ellis11b}.  However, there are no widely available effective computational tools to compute the Tutte polynomial of a general graph of reasonable size.

In this paper, we follow a combinatorial approach and use the self-similar structure to investigate the Tutte polynomial of a family of recursive graphs which is called Farey graphs. The Farey graph was first introduced by Matula and Kornerup in 1979 and further studied by Colbourn in 1982. Recently, this graph was used as a deterministic network model by Zhang and Comellas~\cite{Zhang11,Zhang12}. This network model exhibits some remarkable properties of real networks, it is small-world with its average distance increasing logarithmically with its vertex number, and its clustering coefficient converges to a large constant $\ln 2$~\cite{Zhang11}. The Farey graph also has many interesting graphical properties, such as it is minimally 3-colorable, uniquely Hamiltonian, maximally outerplanar and perfect, see \cite{Matula79,Colbourn82,Biggs88}. It is worth mentioning that the Potts model on recursively defined graphs was studied first by Dhar \cite{Dhar77} and some related work can be found in Refs.~\cite{Alvarez12,Simoi09}.

By analyzing all the spanning subgraphs of this family of graphs, we give recursive formulas for the Tutte polynomial (see Theorem~\ref{th:1} and Corollary~\ref{C1}).
In particular, as special cases of the general Tutte polynomial, we get:
\begin{itemize}
\item The number of spanning trees (see equation~\ref{d11});
\item The number of connected spanning subgraphs (see equation~\ref{d8});
\item The chromatic polynomial $P(G;\lambda)$ (see equation~\ref{d13});
\item The reliability polynomial $R(G;p)$ (see equation~\ref{d12}).
\end{itemize}
\section{Preliminaries}\label{Section preliminare}
In this section we briefly discuss some necessary background that will be used for our calculations. We use standard graph terminology and the words ``network'' and ``graph'' indistinctly.
\subsection{Tutte polynomial}
For a graph $G$ we denote by $V(G)$ its set of vertices and by $E(G)$ its set of edges. Let $k(G)$ be the number of connected components of the graph $G$. There is a useful relation that expresses the Tutte polynomial $T(G;x,y)$ as a sum of contributions from spanning subgraphs of $G$. Here, a spanning subgraph $H=(V(H),E(H))$ has the same vertex set as $G$ and a subset of $E(G)$, $E(H)\subseteq E(G)$. This relation is
\begin{equation}\label{a1}
T(G;x,y)= \sum_{H\subseteq G}(x-1)^{r(G)-r(H)}(y-1)^{n(H)},
\end{equation}
where $r(H)=|V(H)|-k(H)$ is the rank of $H$ and $n(H)=|E(H)|-r(H)$ is the nullity of $H$. Recall that a one-point join $G*H$ of two graphs $G$ and $H$ is formed by identifying a vertex $u$ of $G$ and a vertex $v$ of $H$ into a single vertex $w$ of $G*H$. Using eq.~(\ref{a1}), it is easy to show that the Tutte polynomial fulfills the following property:
\begin{equation}\label{P1}
T(G*H;x,y)=T(G;x,y)T(H;x,y).
\end{equation}
We will refer to this equality as Property~\ref{P1} hereafter.
In the sequel of the paper, we will be interested in special evaluations of the Tutte polynomial at some particular points $(x,y)$, that allow us to deduce many combinatorial and algebraic properties of the graphs considered. We first recall some definitions. A spanning forest of a graph $G$ is a spanning subgraph of $G$ which is a forest. A connected spanning subgraph of a graph $G$ is a spanning subgraph of $G$ that is connected. The special evaluations of interest are (i) $T(G;1,1)=N_{ST}(G)$, the number of spanning trees of $G$; (ii) $T(G;2,1)=N_{SF}(G)$, the number of spanning forests of $G$; (iii) $T(G;1,2)=N_{CSSG}(G)$, the number of connected spanning subgraphs of $G$ \cite{Ellis11a}.

The Tutte polynomial also contains several other graphical polynomials as partial evaluations, such as the reliability polynomial, the flow polynomial and the chromatic polynomial. For a connected graph $G=(V,E)$, let $R(G;p)$ be the reliability polynomial of $G$. Considering a related graph $H$ obtained
from $G$ by going through the full edge set of $G$ and, for each edge, randomly
retaining it with probability $p$ (thus deleting it with probability $1-p$),
where $0 \le p \le 1$.  The reliability polynomial $R(G;p)$
gives the resultant probability that any two vertices in $H$ are connected,
i.e., that for any two such vertices, there is a path between them consisting
of a sequence of connected edges of $G$ \cite{Ellis11a}. The chromatic polynomial $P(G;\lambda)$ of a graph $G$ is an important specialization of the Tutte polynomial, which counts the number of ways of coloring the vertices of $G$ subject to the constraint that no adjacent pairs of vertices have the same color \cite{Read68,Whitney32}. In a different form, the chromatic polynomials arise in statistical physics as a special limiting case, namely the zero-temperature limit of the anti-ferromagnetic Potts model, and in that context their complex roots are of particular interest \cite{Chang03,Salas01}. Their connections with the Tutte polynomial are given by the following equations.
\begin{eqnarray}
R(G;p)&=&q^{n(G)}p^{r(G)}T(G;1,q^{-1}),\label{a3}\\
P(G;\lambda)&=&(-1)^{r(G)}\lambda ^{k(G)}T(G;1-\lambda,0),\label{a4}
\end{eqnarray}
where $q=1-p$, $n(G)$ and $r(G)$ are the nullity and the rank of $G$.
\subsection{Farey graphs}
The Farey graph \cite{Matula79,Colbourn82,Zhang11,Zhang12} is derived from the famous Farey sequence \cite{Hardy79} and this graph can be created in the following recursive way. Let  $G_n$ denote the Farey graph after $n$ $(n \geqslant 0)$ generations. There are two special vertices in $G_n$ which are represented by $X_n$ and $Y_n$, respectively. Initially $n=0$, $G_0$ is an edge connecting two special vertices. For $n \geqslant 1$, $G_n$ can be obtained by joining two copies of $G_{n-1}$ which are labeled by $G_{n-1}^1$ and $G_{n-1}^2$. The special vertices of $G_{n-1}^i$ are represented respectively by $X_{n-1}^i$ and $Y_{n-1}^i$, for $i=1,2$. In the constructing process, $X_{n-1}^2$ and $Y_{n-1}^1$ are identified as a new vertex $Z_n$ of $G_n$, $X_{n-1}^1$ and $Y_{n-1}^2$ are connected by a new edge $e_n$. Moreover, $X_{n-1}^1$ and $Y_{n-1}^2$ become respectively $X_n$ and $Y_n$, see fig.~\ref{F1}.
\vspace{1cm}
\begin{figure}[h]
\begin{center}
\includegraphics[width=0.80\linewidth]{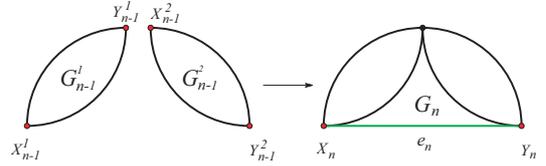}
\end{center}
\caption{Recursive construction of Farey graphs that highlights their
self-similarity}.
\label{F1}
\end{figure}
\begin{figure}[h]
\begin{center}
\includegraphics[width=0.60\linewidth]{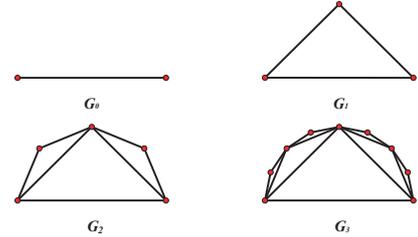}
\end{center}
\caption{Farey graphs $G_0$ to $G_3$}.
\label{12}
\end{figure}

It is easy to prove by induction that the order and size of the Farey graph $G_n=(V_n,E_n)$ are, respectively,
\begin{equation}\label{b1}
|V_n|=2^{n}+1, \quad |E_n|=2^{n+1}-1.
\end{equation}
\begin{os}
There is another definition of the Farey graph given in \cite{Zhang11} as Definition 2.1 which seems simpler and easier to imagine. In fact, the Farey graph $G_n=(V_n,E_n)$ can be constructed in the following iterative way:
For $n=0$, $G_0$ has two vertices and an edge joining them.
For $n \geqslant 1$, $G_n$ is obtained from $G_{n-1}$ by adding to every edge introduced at step $n-1$ a new vertex adjacent to the endvertices of this edge, see fig.~\ref{12}.
\end{os}

\section{The Tutte polynomial of the Farey graph $G_n$ }
First of all, we define the following partition of the set of spanning subgraphs of $G_n$:
\begin{itemize}
  \item $G_{1,n}$ denotes the set of spanning subgraphs of $G_n$, where the special vertices of $G_n$ belong to the same connected component;
  \item $G_{2,n}$ denotes the set of spanning subgraphs of $G_n$, where the special vertices of $G_n$ belong to two different  connected components.
\end{itemize}
 \begin{figure}[htbp]
 \centering
 \includegraphics[width=0.2\textwidth]{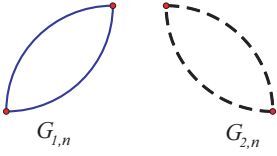}
 \caption{Schematic illustration of the two types of spanning subgraphs $G_{1,n}$, $G_{2,n}$ derived from $G_n$. In this figure, red vertices represent the special vertices of $G_n$ and two special vertices joined by two blue lines belong to the same connected component. Dotted lines means that the two special vertices belong to different connected components}.
 \label{F2}
 \end{figure}
Fig.~\ref{F2} illustrates schematically these two types of spanning subgraphs.
\begin{table*}[!t]
\centering
\caption{All combinations and corresponding contributions}
\label{T1}
\begin{tabular}{ccccc|ccccc}
\toprule[1pt]
$H_{n-1}^a*H_{n-1}^b$& $C$& $H_n$& Contribution&Graphic& $H_{n-1}^a*H_{n-1}^b$& $C$& $H_n$& Contribution& \\
\midrule
($G_{1,n-1},G_{1,n-1}$)&$\{e_{n}\}$&$G_{1,n}$&$(y-1)T_{1,n-1}^2$&\includegraphics[scale=0.1]{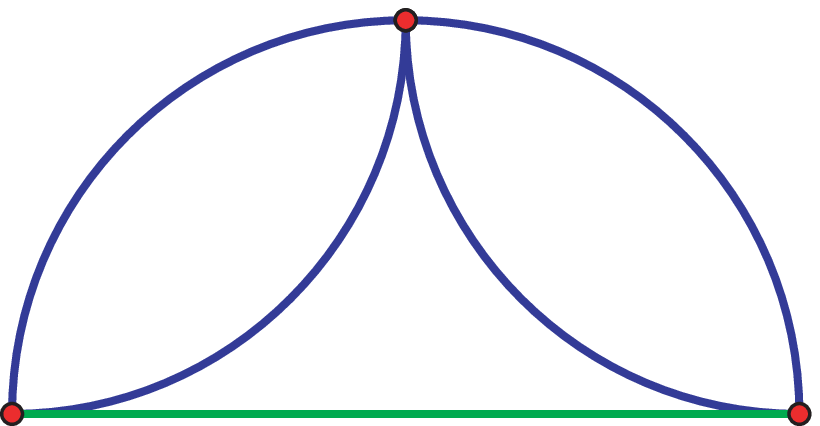}&($G_{1,n-1},G_{1,n-1}$)&$\emptyset$&$G_{1,n}$&$T_{1,n-1}^2$&\includegraphics[scale=0.1]{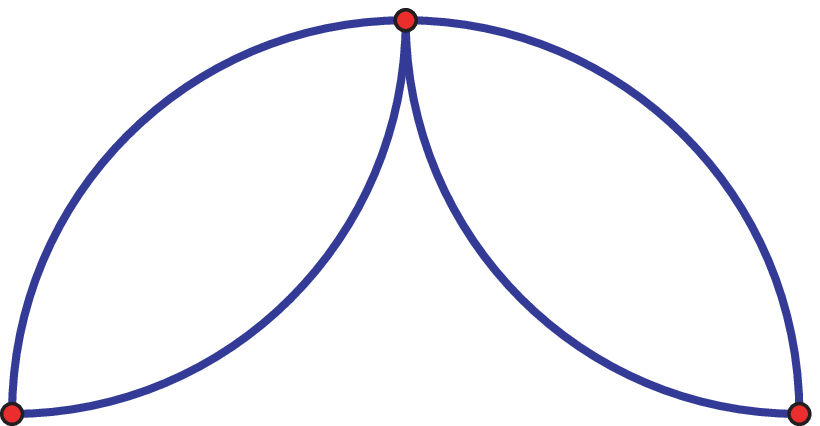}\\
($G_{1,n-1},G_{2,n-1}$)&$\{e_{n}\}$&$G_{1,n}$&$\frac{1}{x-1}T_{1,n-1}T_{2,n}$&\includegraphics[scale=0.1]{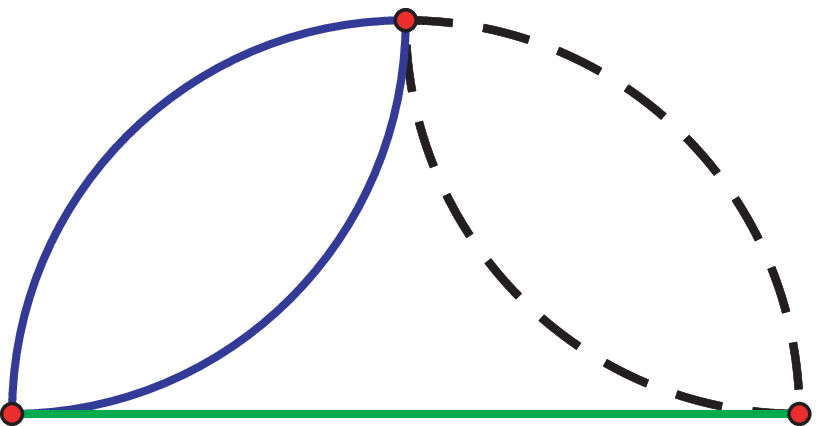}&($G_{1,n-1},G_{2,n-1}$)&$\emptyset$&$G_{2,n}$&$T_{1,n-1}T_{2,n-1}$&\includegraphics[scale=0.1]{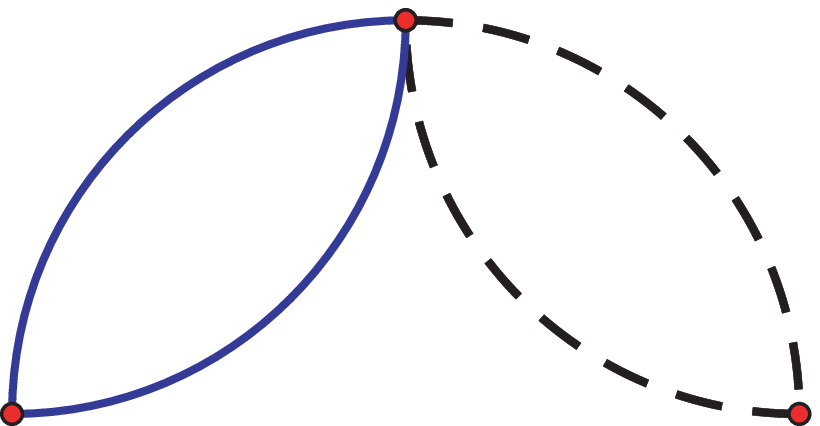}\\
($G_{2,n-1},G_{1,n-1}$)&$\{e_{n}\}$&$G_{1,n}$&$\frac{1}{x-1}T_{1,n-1}T_{2,n}$&\includegraphics[scale=0.1]{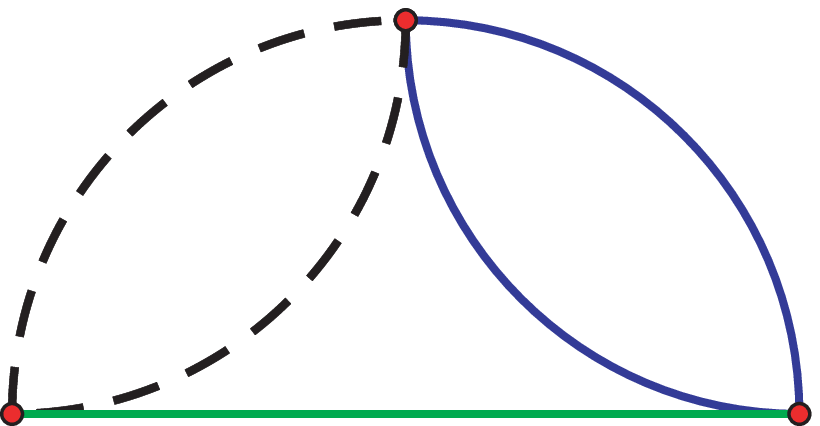}&($G_{2,n-1},G_{1,n-1}$)&$\emptyset$&$G_{2,n}$&$T_{1,n-1}T_{2,n-1}$&\includegraphics[scale=0.1]{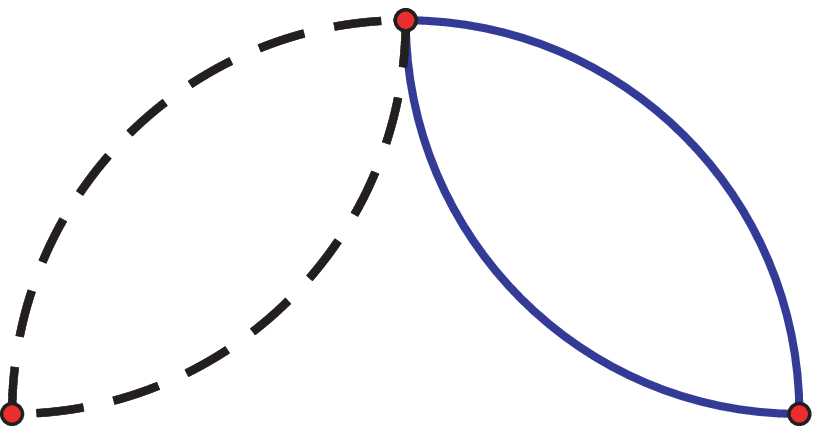}\\
($G_{2,n-1},G_{2,n-1}$)&$\{e_{n}\}$&$G_{1,n}$&$\frac{1}{x-1}T_{2,n-1}^2$&\includegraphics[scale=0.1]{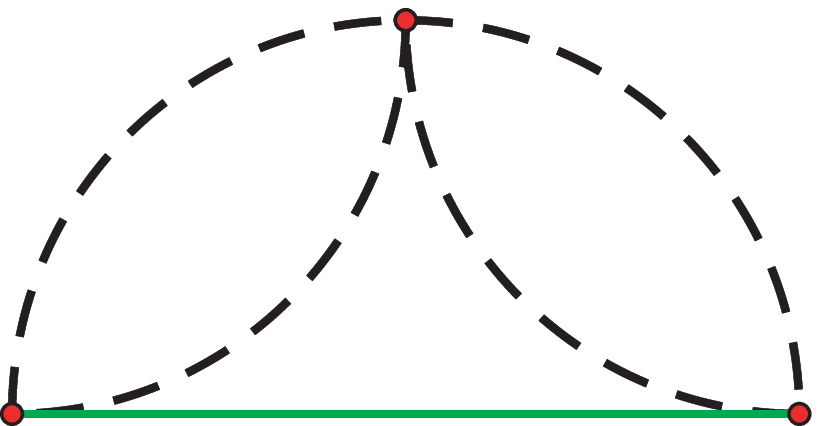}&($G_{2,n-1},G_{2,n-1}$)&$\emptyset$&$G_{2,n}$&$T_{2,n-1}^2$&\includegraphics[scale=0.1]{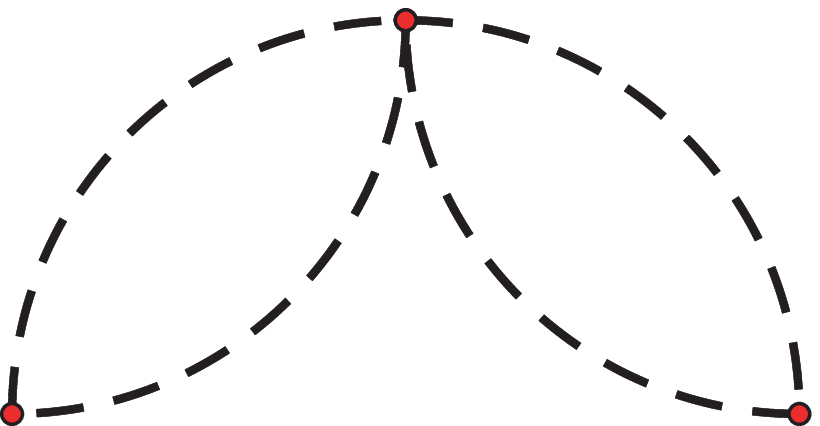}\\
\bottomrule[1pt]
\end{tabular}
\end{table*}

Observe that, for each $n \geqslant 0$, we have the partition $G_{1,n} \cup G_{2,n}$
of the set of spanning subgraphs of  $G_n$. Next let us simply denote by $T_n(x,y)$ the Tutte polynomial $T(G_n;x,y)$ and define, for every $n \geqslant 1$, the following polynomials:
\begin{itemize}
  \item $T_{1,n}(x,y)= \sum_{H \in G_{1,n}}(x-1)^{r(G_n)-r(H)}(y-1)^{n(H)}$;
  \item $T_{2,n}(x,y)=\sum_{H \in G_{2,n}}(x-1)^{r(G_n)-r(H)}(y-1)^{n(H)}$.
\end{itemize}
We have
\begin{equation}\label{c1}
T_n(x,y)=T_{1,n}(x,y)+T_{2,n}(x,y).
\end{equation}

In order to give a recursive formula for $T_n(x,y)$, we provide recursive formulas for $T_{1,n}$ and $T_{2,n}$, respectively.

 From the construction of $G_n$, it is not difficult to observe that any spanning subgraph of $G_{n}$ can be obtained from two spanning subgraphs of $G_{n-1}$, with  possibly an edge. Let $H_{n-1}^a$ , $H_{n-1}^b$ be two spanning subgraphs of $G_{n-1}$. We denote $H_{n-1}^a*H_{n-1}^b$ as the graph obtained from $H_{n-1}^a$ and $H_{n-1}^b$ by identifying the vertex $Y_{n-1}$ of $H_{n-1}^a$ and the vertex $X_{n-1}$ of $H_{n-1}^b$. Let $H_n=H_{n-1}^a*H_{n-1}^b \cup C$, where $C$ may be $\{e_n\}$ or an empty set $\emptyset$. Thus $H_n$ is a spanning subgraph of $G_n$. Let $(G_{i,n-1},G_{j,n-1})$ be the set of $H_{n-1}^a*H_{n-1}^b$  with $H_{n-1}^a \in G_{i,n-1}$ and  $H_{n-1}^b \in G_{j,n-1}$, for $i, j=1, 2$.

\begin{teo}
\label{th:1}
For each $n \geqslant 1$, the Tutte polynomial $T_n(x,y)$ of $G_n$
is given by
$$
T_n(x,y)=T_{1,n}(x,y)+T_{2,n}(x,y),
$$
where the polynomials $T_{1,n}(x,y)$, $T_{2,n}(x,y)$ satisfy the following recursive
relations:
\begin{eqnarray*}
T_{1,n}(x,y)&=& yT_{1,n-1}^2 +\frac{2}{x-1}T_{1,n-1}T_{2,n-1}+\frac{1}{x-1}T_{2,n-1}^2, \\
T_{2,n}(x,y)&=& T_{2,n-1}^2+2T_{1,n-1}T_{2,n-1},
\end{eqnarray*}
with initial conditions
\begin{equation*}
T_{1,0}(x,y)=1,\quad T_{2,0}(x,y)=x-1.
\end{equation*}
\end{teo}

\begin{proof}
The initial conditions are easily verified. The strategy of the proof is to study all the possible combinations of $H_{n-1}^a*H_{n-1}^b$ and $C$, and analyze which kind of contribution they give to $T_{1,n}(x,y)$ and  $T_{2,n}(x,y)$.

As shown in table~\ref{T1}, we have,
\begin{eqnarray*}
T_{1,n}(x,y)&=& \sum_{H_n \in G_{1,n}}(x-1)^{r(G_n)-r(H)}(y-1)^{n(H)} \\
            &=& yT_{1,n-1}^2 +\frac{2}{x-1}T_{1,n-1}T_{2,n-1}+\frac{1}{x-1}T_{2,n-1}^2,\\
T_{2,n}(x,y)&=& \sum_{H_n \in G_{2,n}}(x-1)^{r(G_n)-r(H)}(y-1)^{n(H)} \\
            &=& 2T_{1,n-1}T_{2,n-1}+T_{2,n-1}^2.
\end{eqnarray*}
\end{proof}

It can be  proven by induction that $x-1$ divides $T_{2,n}$ for each $n \geqslant 0 $. As a consequence, we can write $T_{2,n}(x,y)=(x-1)N_n(x,y)$, with $N_n(x,y) \in \mathbb{Z}[x,y]$. Then we have the following corollary which will be very useful in the following derivation of special evaluations of the Tutte polynomial $T_n(x,y)$.
\begin{cor}\label{C1}
For each $n\geqslant 1$, the Tutte polynomial $T_n(x,y)$ of $G_n$
is given by
$$
T_n(x,y)=T_{1,n}(x,y)+(x-1)N_n(x,y),
$$
where the polynomials $T_{1,n}(x,y)$, $N_n(x,y)$ satisfy the following recursive
relations:
\begin{eqnarray*}
T_{1,n}(x,y)&=& yT_{1,n-1}^2 +2T_{1,n-1}N_{n-1}+(x-1)N_{n-1}^2,
\end{eqnarray*}
\begin{eqnarray*}
N_{n}(x,y)=2T_{1,n-1}N_{n-1}+(x-1)N_{n-1}^2
\end{eqnarray*}
with initial conditions
$$T_{1,0}(x,y)=1, \quad N_0(x,y)=1.$$
\end{cor}
\section{Simple applications}
The problem of spanning trees is closely related to various aspects of networks, such as dimer coverings \cite{Tseng03} and random walks \cite{Dhar97,Noh04}. Thus it is of great interest to determine the exact number of spanning trees $N_{ST}(G)$ \cite{Zhang12,Lin11,Zhang10a}.  Two well-known methods for computing $N_{ST}(G)$ are as follows: (i) via the Laplacian matrix \cite{Biggs93} and (ii) as a special value of the Tutte polynomial.  With these obtained results, we go on to compute $N_{ST}(G_n)$ using the second method. First, according to Corollary~\ref{C1}, let $x=1$, we have $T_n(1,y)=T_{1,n}(1,y)$, and
\begin{eqnarray}
T_{1,n}(1,y) &=& yT_{1,n-1}^2+2T_{1,n-1}N_{n-1},\label{d1}\\
N_n(1,y)&=&2T_{1,n-1}N_{n-1}.\label{d2}
\end{eqnarray}
Eqs.~(\ref{d1}) and~(\ref{d2}) together yield a useful relation given by
\begin{equation}\label{d3}
\frac{T_{n+1}(1,y)}{T_n^2(1,y)}=(2+y)-\frac{2yT_{n-1}^2}{T_n}.
\end{equation}
For $t \geqslant 2$, we introduce
 \begin{equation}\label{d4}
    a_t(1,y)=\frac{T_t(1,y)}{T_{t-1}(1,y)^2}, \ \hbox{with} \  a_1(1,y)=2+y.
    \end{equation}
Thus, eq.~(\ref{d3}) can be further written as
\begin{equation}\label{d5}
a_{n+1}(1,y)=(2+y)-\frac{2y}{a_n(1,y)}.
\end{equation}

If $y=2$, eq.~(\ref{d5}) leads to
\begin{equation}\label{d6}
\frac{1}{a_{n+1}-2}=\frac{1}{a_{n}-2}+\frac{1}{2}.
\end{equation}
Since $a_1(1,2)=4$, we have
\begin{equation}\label{d7}
a_n(1,2)=\frac{2(n+1)}{n}.
\end{equation}
Substituting the above obtained expression of $a_n(1,2)$ into eq.~(\ref{d4}) and using the initial condition $T_0(1,2)=1$ yields
\begin{equation}\label{d8}
      T_n (1,2)  =\frac{2(n+1)}{n}T_{n-1}^2 =2^{2^n-1}(n+1)\prod_{i=2}^ni^{2^{n-i}}.
     \end{equation}
It is well known that $N_{CSSG}(G_n)=T_n(1,2)$. Therefore, we obtain explicit expression of the number of connected spanning subgraphs of the Farey graph $G_n$.

If $y \neq 2$, eq.~(\ref{d5}) can be solved to obtain
\begin{equation}\label{d9}
  a_n(1,y)=\frac{2^{n+1}-y^{n+1}}{2^n-y^n}
  \end{equation}
Using the initial condition $T_0(1,y)=1$ and the expression for $a_n(1,y)$ provided by eq.~(\ref{d9}), eq.~(\ref{d4}) can be solved inductively to obtain
\begin{eqnarray}\label{d10}
      T_n (1,y)  &=& \frac{2^{n+1}-y^{n+1}}{2^n-y^n}T_{n-1}^2 \nonumber\\
            &=&\frac{2^{n+1}-y^{n+1}}{(2-y)^{2^{n-1}}}\prod_{i=2}^n(2^i-y^i)^{2^{n-i}}.
     \end{eqnarray}
     Notice that, the eq.(\ref{d10})
     is also right when $y=2$, since
     \begin{equation*}
     2^i-y^i=(2-y)(2^{i-1}+ 2^{i-2}y+\dots+2^{i-k-1}y^k+\dots+y^{i-1})
     \end{equation*}
      and
      \begin{equation*}
    (2-y)\prod_{i=2}^n(2-y)^{2^{n-i}}=(2-y)^{2^{n-1}}.
     \end{equation*}

     Let $y=1$, eq.~(\ref{d10}) leads to
     \begin{equation}\label{d11}
     N_{ST}(G_n)=T_n(1,1)=(2^{n+1}-1)\prod_{i=2}^n(2^i-1)^{2^{n-i}}.
     \end{equation}
      Thus far, we have derived the exact number of spanning trees of the Farey graph $G_n$. Note that this specialization of $T_n(x,y)$ returns the equation obtained in Ref.~\cite{Zhang12}, where the authors use the Laplacian matrix.

     In general case, the calculation of the reliability polynomial is NP-hard \cite{Ball80}. But for the Farey graph,  inserting eqs.~(\ref{d10}) and ~(\ref{b1}) into eq.~(\ref{a3}), we obtain the explicit expression for the the reliability polynomial $R(G_n;p)$ as
     \begin{equation}\label{d12}
     R(G_n;p)=q^{2^n-1}p^{2^n}\frac{2^{n+1}-q^{-n-1}}{(2-q^{-1})^{2^{n-1}}}\prod_{i=2}^n(2^i-q^{-i})^{2^{n-i}},
     \end{equation}
     where $q=1-p$.

     Finally, we will determine the chromatic polynomial $P(G_n;\lambda)$. Since $P(G_n;\lambda)$ satisfies the following relation:
     \begin{eqnarray}\label{d13}
     P(G_n;\lambda)&=&(-1)^{|V_n|-k(G_n)}\lambda^{k(G_n)}T_n(1-\lambda,0) \nonumber\\
                   &=& \lambda T_n(1-\lambda,0),
     \end{eqnarray}
    we turn our aim to finding $T_n(x,0)$. From Corollary~\ref{C1}, let $y=0$, we can easily get the following recursive equations,
    \begin{eqnarray}\label{d14}
    T_n(x,0)&=&T_{1,n}(x,0)+(x-1)N_n(x,0), \label{d14}\\
    T_{1,n}(x,0)&=&2T_{1,n-1}N_{n-1}+ (x-1)N_{n-1}^2, \label{d15}\\
    N_n(x,0)&=&2T_{1,n-1}N_{n-1}+(x-1)N_{n-1}^2.\label{d16}
    \end{eqnarray}
    It follows from eqs.~(\ref{d15}) and~(\ref{d16}) that
    \begin{equation}\label{d17}
    T_{1,n}(x,0)=N_n(x,0)
    \end{equation}
    Substituting eq.~(\ref{d17}) into eq.~(\ref{d16}), we obtain
    \begin{equation}\label{d18}
    N_n(x,0)=(x+1)N_{n-1}^2(x,0).
    \end{equation}
    Considering the initial condition $N_0(x,0)=1$, eq.~(\ref{d14}) is solved to yield
    \begin{equation}\label{d19}
    T_n(x,0)=xN_n(x,0)=x(x+1)^{2^n-1}
    \end{equation}
    Plugging the last expression into eq.~(\ref{d13}), we arrive at the explicit formula for the chromatic polynomial,
    \begin{equation}\label{d20}
    P(G_n;\lambda)=\lambda(1-\lambda)(2-\lambda)^{2^n-1}.
    \end{equation}
    Then, it is easy to see that $G_n$ is minimally 3-colorable.
\section{conclusions}
In this paper, we find recursive formulas for the Tutte polynomial of the Farey graph family by using a method, based on their self-similar structure. We also evaluate special cases of these results to compute the corresponding reliability polynomial, chromatic polynomial, the number of spanning trees and the number of connected spanning subgraphs.
\begin{acknowledgments}
The authors  are supported by the National Natural Science Foundation of China Under grant No. 1171102
\end{acknowledgments}

\appendixpage

\section{Tutte polynomial of other self-similar networks}
The computation of the Tutte polynomial of a general graph is NP-hard. Indeed, even evaluating $T(G;x,y)$ for specific values $(x,y)$ is \#P-complete, except for a few special points and a special hyperbola \cite{Jaeger90}. However we gain explicit expressions of the Tutte polynomial of some other self-similar networks with different degree distributions, exponential or scale-free.

\subsection{Koch network}

We first study the self-similar Koch network family \cite{Bin12,Zhang10}. The Koch networks are derived from the Koch fractals \cite{Schneider65} and can be constructed recursively. Let $K_{m,n}$ ($m$ is a natural number) be the Koch network in generation $n$. Then, the family of Koch networks can be generated following a recursive-modular method.
For $n=0$, $K_{m,0}$ consists of a triangle with three vertices labeled respectively by $X$, $Y$ and $Z$, called hub vertices, which have the highest degree among all vertices in the networks.
For $n \geqslant 1$, $K_{m,n}$ is obtained from $3m+1$ copies of $K_{m,n-1}$ by joining them at the hub vertices. As shown in Fig.~\ref{F3}, the network $K_{m,n+1}$ may be obtained by the juxtaposition of $3m+1$ copies of $K_{m,n}$ which are labeled as $K_{m,n}^1,K_{m,n}^2,\dots,K_{m,n}^{3m}$ and $K_{m,n}^{3m+1}$, respectively.


\begin{figure}[htbp]
 \centering
 \includegraphics[width=0.4\textwidth]{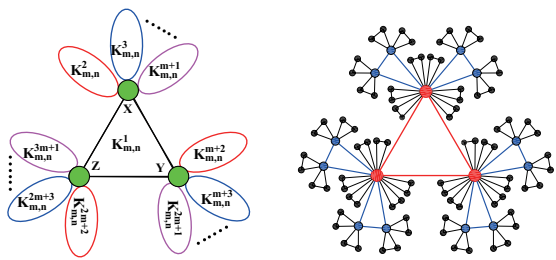}
 \caption{Recursive construction of Koch networks that highlights their
self-similarity. The right one is a particular Koch network $K_{2,2}$.}
\label{F3}
 \end{figure}
The number of vertices and edges of $K_{m,n}=(V_n,E_n)$ are, respectively, $|V_n|=2(3m+1)^n+1$ and $|E_n|=3(3m+1)^n$. The Koch networks present some typical properties of real-world networks. They are scale-free with vertex degree distribution $P(k)$ in a power-law form, $P(k) \thicksim k^{-\gamma}$, where $\gamma=1+\frac{ln(3m+1)}{ln(m+1)}$. They also have an obvious small-world characteristic. They average path length increases logarithmically with the network size and their clustering coefficient is very high.

We now begin to determine the Tutte polynomial $T(K_{m,n};x,y)$. According to the construction of the Koch network, $K_{m,n}$ can be obtained from $3m+1$ copies of $K_{m,n-1}$ through one-point join.
From Property~\ref{P1}, we have
 \begin{equation}\label{e1}
T(K_{m,n};x,y)=T(K_{m,n-1};x,y)^{3m+1}.
\end{equation}
Considering the initial condition $T(K_{m,0};x,y)=x^2+x+y$, we obtain explicit expression of the Tutte polynomial of the Koch network $K_{m,n}$ as,
\begin{equation}\label{e2}
T(K_{m,n};x,y)=(x^2+x+y)^{(3m+1)^n}.
\end{equation}
We next discuss some special valuations of $T(K_{m,n};x,y)$. It is easy to compute that
\begin{eqnarray}
N_{ST}(K_{m,n})    &=&T(K_{m,n};1,1)=3^{(3m+1)^n};\label{e3}\\
N_{CSSG}(K_{m,n})  &=&T(K_{m,n};1,2)=4^{(3m+1)^n};\label{e4}\\
N_{SF}(K_{m,n})    &=&T(K_{m,n};2,1)=7^{(3m+1)^n}.\label{e5}
\end{eqnarray}
Note that these quantities coincide with the result obtained in Refs.~\cite{Bin12,Zhang10} without using the Tutte polynomial.

\subsection{Small-world exponential network}
In this section we discuss a self-similar network which is observed from some real-life systems. This network is constructed iteratively \cite{Barriere09}, denoted by $S_{n}$ after $n \ (n \geqslant 0)$ iterations. The network starts from a triangle $S_{0}$, with three vertices called initial vertices. For $n \ge 1$, $S_{n}$ is obtained from $S_{n-1}$: for each existing vertex in $S_{n-1}$, two new vertices are generated, which and their mother vertex together form a new triangle, see Fig.~\ref{F4}. It is easy to derive that the number of vertices and edges of $S_{n}$ are, respectively, $|V_n|=3^{n+1}$ and $|E_n|=3(3^{n+1}-1)/2$. The resultant network has a degree distribution decaying exponentially with the degree. In addition, it is small-world with the average distance growing logarithmically with its size.

As shown in Fig.~\ref{F4}, $S_n$ can be obtained from three copies of $S_{n-1}$ and a triangle through one-point join. Since the Tutte polynomial of a triangle is equal to $x^2+x+y$, we have
\begin{equation}\label{e6}
T(S_n;x,y)=(x^2+x+y)T^3(S_{n-1};x,y).
\end{equation}
Additionally, it is easy to know that $T(S_0;x,y)=x^2+x+y$. We can therefore obtain  the exact expression of the Tutte polynomial of $S_n$,
\begin{equation}\label{e7}
T(S_n;x,y)=(x^2+x+y)^{(3^{n+1}-1)/2}.
\end{equation}
Based on the above equation, we can show some interesting quantities for the network, such as the number of spanning trees $N_{ST}(S_n)$, the number of connected spanning subgraphs $N_{CSSG}(S_n)$ and the number of spanning forest $N_{SF}(S_n)$,
\begin{figure}[htbp]
 \centering
 \includegraphics[width=0.4\textwidth]{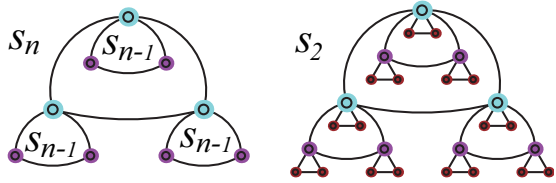}
\caption{Recursive construction of the small-world network family that highlights their
self-similarity. The right one is a particular network $S_2$.}
\label{F4}
\end{figure}
\begin{eqnarray}
N_{ST}(S_n)&=&3^{(3^{n+1}-1)/2};\label{e8}\\
N_{CSSG}(S_n)&=&4^{(3^{n+1}-1)/2};\label{e9}\\
N_{SF}(S_n)&=&7^{(3^{n+1}-1)/2}.\label{e10}
\end{eqnarray}
The same result about the number of spanning trees has been proven in Ref.~\cite{Lin11}, where the authors use the Laplacian matrix.

\acknowledgments

\end{document}